\documentclass[11pt, a4paper]{article}

\usepackage{amsthm}
\usepackage{mathrsfs}
\usepackage{amsmath,amssymb,latexsym,color}
\usepackage[colorlinks,
linkcolor=blue,
anchorcolor=green,
citecolor=magenta
]{hyperref}
\usepackage{graphicx}
\usepackage{tikz}
\usetikzlibrary{calc}
\usepackage{CJK}

\usepackage{authblk}

\long\def\delete#1{}

\usepackage{color}

\definecolor{Blue}{rgb}{0,0,1}
\definecolor{Red}{rgb}{1,0,0}
\definecolor{DarkGreen}{rgb}{0,0.6,0}
\definecolor{DarkYellow}{rgb}{1,1,0.2}
\definecolor{DarkPurple}{rgb}{.6,0,1}

\usepackage{xcolor}
\usepackage[normalem]{ulem}

\usepackage{cleveref}
\crefformat{section}{\S#2#1#3}
\crefformat{subsection}{\S#2#1#3}
\crefformat{subsubsection}{\S#2#1#3}
\crefrangeformat{section}{\S\S#3#1#4 to~#5#2#6}
\crefmultiformat{section}{\S\S#2#1#3}{ and~#2#1#3}{, #2#1#3}{ and~#2#1#3}

\def\mma{\mathscr{C}}
\def\ma{\mathscr{A}}

\def\mf{\mathscr{F}}
\def\mg{\mathscr{G}}

\def\ml{\mathscr{L}}
\def\mm{\mathscr{M}}
\def\mo{\mathscr{M}}

\def\mr{\mathscr{R}}
\def\ms{\mathscr{S}}

\def\mmu{\mathscr{U}}

\def\bs{\setminus}

\def\ag{AG(n,\ff)}
\def\ff{\mathbb{F}_q}

\def\ge{\geqslant}
\def\le{\leqslant}
\def\b{\brack}

\def\kn{\mm(k,n)}
\def\ro{\romannumeral}
\def\rmo{\rm{(1)}}
\def\rmt{\rm{(2)}}
\def\rmth{\rm{(3)}}

\numberwithin{equation}{section}

\newtheorem{thm}{Theorem}[section]
\newtheorem{lem}[thm]{Lemma}

\begin{document}
	\setcounter{page}{1}
	\renewcommand{\thefootnote}{}
	\newcommand{\remark}{\vspace{2ex}\noindent{\bf Remark.\quad}}
	\renewcommand{\abovewithdelims}[2]{%
		\genfrac{[}{]}{0pt}{}{#1}{#2}}

	
	\def\qed{\hfill$\Box$\vspace{11pt}}
	
	\title {\bf  Cross $t$-intersecting families for finite affine spaces}

	\author{Tian Yao\thanks{E-mail: \texttt{yaotian@mail.bnu.edu.cn}}}
	\author{Kaishun Wang\thanks{Corresponding author. E-mail: \texttt{wangks@bnu.edu.cn}}}
	\affil{Laboratory of Mathematics and Complex Systems (Ministry of Education), School of
		Mathematical Sciences, Beijing Normal University, Beijing 100875, China}

	\date{}
	
	\openup 0.5\jot
	\maketitle

	\begin{abstract}
		
		Denote the collection of all $k$-flats in $\ag$ by $\kn$.
		Let $\mf_1\subset\mm(k_1,n)$ and $\mf_2\subset\mm(k_2,n)$ satisfy $\dim(F_1\cap F_2)\ge t$ for any $F_1\in\mf_1$ and $F_2\in\mf_2$. We say they are cross $t$-intersecting families.  Moreover, we say they are trivial if each member of them contains a fixed $t$-flats in $\ag$. In this paper, we show that cross $t$-intersecting families with maximum product of sizes are trivial. We also describe the structure of non-trivial $t$-intersecting families with maximum product of sizes.
		
		\vspace{2mm}
		
		
		\
		
		
		
	\end{abstract}
\section{Introduction}

Intersection problems originate from the famous Erd\H{o}s-Ko-Rado Theorem \cite{EKR}. In recent years, intersection problems for mathematical objects which are relative to vector spaces have been caught lots of attention \cite{VHM1,VHM2,AHM1,AEKR,pepe-ekr-polar-maximal}.

Let $n$ and $k$ be positive integers with $n\ge k$, $V$ an $n$-dimensional vector space over the finite field $\ff$, where $q$ is a prime power, and ${V\b k}_q$ denote the family of all $k$-dimensional subspaces of $V$. We usually replace ``$k$-dimensional subspace'' with ``$k$-subspace'' for short. Define the \emph{Gaussian binomial coefficient} by
$${n\b k}_q:=\prod_{0\le i<k}\dfrac{q^{n-i}-1}{q^{k-i}-1},$$
and set ${n\b0}_q=1$. Note that the size of ${V\b k}_q$ is ${n\b k}_q$. From now on, we will omit the subscript $q$.

Let $t$ be a positive integer. A family $\mf\subset{V\b k}$ is called \emph{$t$-intersecting} if $\dim(F_1\cap F_2)\ge t$ for any $F_1,F_2\in\mf$. A $t$-intersecting family $\mf$ is called \emph{trivial} if there exists a $t$-subspace contained in each element of $\mf$. The Erd\H{o}s-Ko Rado Theorem for vector space \cite{VEKR2,VEKR1,VEKR3} shows that a $t$-intersecting subfamily of ${V\b k}$ with maximum size is trivial when $\dim V>2k$. The structure of non-trivial $t$-intersecting subfamily of ${V\b k}$ with maximum size was determined via the parameter ``$t$-covering number", see \cite{VHM1,VHM2}. For $\mf_1\in{V\b k_1}$ and $\mf_2\subset{V\b k_2}$, we say they are cross $t$-intersecting if $\dim(F_1\cap F_2)$ holds for any $F_1\in\mf_1$ and $F_2\in\mf_2$. Recently, Cao et al \cite{CAO} describe the structure of cross $t$-intersecting families with the first and second larges product of sizes.

Suppose $P$ is a $k$-subspace of $\ff^n$. A coset of  $\ff^n$ relative to $P$ is called a $k$-flat. The \emph{dimension} of a $k$-flat $P+x$ is defined to be the dimension of $P$. A flat $F_1$ is said to be \emph{incident} with a flat $F_2$ if $F_1\subset F_2$ or $F_2\subset F_1$. The point set of $\ff^n$ with all flats and the incidence relation among them is called the $n$-dimensional \emph{affine space}, denoted by $AG(n,\ff)$. Let $F_1\cap F_2$ and $F_1\vee F_2$ denote the intersection of $F_1$ and $F_2$, and the minimum flat containing both $F_1$ and $F_2$, respectively. Observe that the intersection of two flats is either a flat or a empty set.

Denote the set of all $k$-flats in $\ag$ by $\mm(k,n)$. A family $\mf\subset\kn$ is said to be $t$-intersecting if $\dim(F_1\cap F_2)\ge t$ for any $F_1,F_2\in\mf$. A $t$-intersecting family is called \emph{trivial} if every element of $\mf$ contains a fixed $t$-flat. Guo and Xu \cite{AEKR} show that the maximum sized $t$-intersecting families are trivial. The structure of maximum sized non-trivial $t$-intersecting families has also been described \cite{TY,AHM1}. 
There are also some results about ``$0$-intersecting families", see \cite{AHM0,AEKR} for more details.

Let $\mf_1\subset\mm(k_1,n)$ and $\mf_2\subset\mm(k_2,n)$ satisfy that $\dim(F_1\cap F_2)\ge t$ for any $F_1\in\mf_1$ and $F_2\in\mf_2$. We say they are \emph{cross $t$-intersecting}. Moreover, they are called \emph{trivial} if each member of them contains a fixed $t$-flat in $\ag$.

The first main result of this paper is the following.
\begin{thm}\label{trivial}
	Let $n$, $k_1$, $k_2$ and $t$ be positive integers with $k_1,k_2\ge t$ and $n\ge k_1+k_2+3$.	Suppose that $\mf_1\subset\mm(k_1,n)$ and $\mf_2\subset\mm(k_2,n)$ are cross $t$-intersecting families with maximum product of sizes. Then there exists a $t$-flat in $\ag$ contained in each member of $\mf_1$ and $\mf_2$. 
\end{thm}
Based on Theorem \ref{trivial}, we get a more general theorem, see Theorem \ref{trivial2}.

For a flat $A$ and a positive integer $a$, write $\mm(a,A)=\{F\in\mm(a,n): F\subset A\}$. Let $M\in\mm(k_2+1,n)$,  $T\in\mm(t,M)$ and $S\in\mm(t+1,n)$. Write
\begin{equation}\label{c1-c4}
	\begin{aligned}
		&\mma_1(M,T;k_1,t)=\{F\in\mm(k_1,n): T\subset F, \dim(F\cap M)\ge t+1\},\\
		&\mma_2(M,T;k_2,t)=\{F\in\mm(k_2,n): T\subset F\}\cup\{F\in\mm(k_2,M): \dim(F\cap T)=t-1\},\\
		&\mma_3(S;k_1)=\{F\in\mm(k_1,n):S\subset F\},\\
		&\mma_4(S;k_2,t)=\{F\in\mm(k_2,n):\dim(F\cap S)\ge t\}.
	\end{aligned}
\end{equation}

Observe that $\mma_1(M,T;k_1,t)$ and $\mma_2(M,T;k_2,t)$ are cross $t$-intersecting families. So are $\mma_3(S;k_1)$ and $\mma_4(S;k_2,t)$. 
Our second main result describe the structure of cross $t$-intersecting families with the second largest product of sizes.
\begin{thm}\label{non-trivial}
	Let $n$, $k_1$, $k_2$ and $t$ be positive integers with $k_1\ge k_2\ge t+1$ and $n\ge k_1+k_2+t+7$.	Suppose that $\mf_1\subset\mm(k_1,n)$ and $\mf_2\subset\mm(k_2,n)$ are non-trivial cross $t$-intersecting families with maximum product of sizes. 
	\begin{itemize}
		\item[\rmo] If $k_2>2t$, then there exist $M\in\mo(k_2+1,n)$ and $T\in\mo(t,M)$ such that
		\begin{itemize}
			\item[\rm{(\ro1)}] $\mf_1=\mma_1(M,T;k_1,t)$, $\mf_2=\mma_2(M,T;k_2,t)$; or
			\item[\rm{(\ro2)}] $k_1=k_2$ and $\mf_1=\mma_2(M,T;k_1,t)$, $\mf_2=\mma_1(M,T;k_2,t)$.
		\end{itemize}
		\item[\rmt] If $k_2\le2t$, then there exists $S\in\mo(t+1,n)$ such that
		\begin{itemize}
			\item[\rm{(\ro1)}] $\mf_1=\mma_3(S;k_1)$, $\mf_2=\mma_4(S;k_2,t)$; or
			\item[\rm{(\ro2)}] $k_1=k_2$ and $\mf_1=\mma_4(S;k_1,t)$, $\mf_2=\mma_3(S;k_2)$.
		\end{itemize}
	\end{itemize}
\end{thm}

\section{Preliminaries}

In this section, we give some useful lemmas in preparation for the proof of our main theorems.
\begin{lem}\label{gaosifangsuo}
	Let $m$ and $i$ be positive integers with $i<m$. Then the following hold.
	\begin{itemize}
		\item[\rmo] $q^{m-i}<\frac{q^m-1}{q^i-1}<q^{m-i+1}$ and $q^{i-m-1}<\frac{q^i-1}{q^m-1}<q^{i-m}$;
		\item[\rmt] $q^{i(m-i)}<{m\b i}<q^{i(m-i+1)}$.
	\end{itemize}
\end{lem}

Let $V$ be a $(e+\ell)$-space over $\ff$ and $W$ a fixed $\ell$-subspace of $V$. A subspace $P\in{V\b m}$ is called an $(m,h)$-type subspace of $V$ if $\dim(P\cap W)=h$. Let $N'(m_1,h_1;m,h;e+\ell,e)$ denote the number of $(m,h)$-type subspaces of $V$ containing a fixed subspace of type $(m_1,k_1)$.
\begin{lem}\label{af-QY}\rm{(\cite[Lemma 2.3]{QY})}
	$N'(m_1,h_1;m,h;e+\ell,e)\neq0$ if and only if $0\le h_1\le h\le l$ and $0\le m_1-h_1\le m-h\le e$. Moreover, if $N'(m_1,h_1;m,h;e+\ell,e)\neq0$, then
	$$N'(m_1,h_1;m,h;e+\ell,e)=q^{(\ell-h)(m-h-(m_1-h_1))}{e-(m_1-h_1)\b(m-h)-(m_1-h_1)}{\ell-h_1\b h-h_1}.$$
\end{lem}
\begin{lem}\label{af-jishu}\rm{(\cite[Theorems 1.18-1.19]{JS})} Let $U$ be an $m$-flat in $\ag$ with $0\le m\le n$.
	\begin{itemize}
		\item[\rmo] If $0\le k\le m$, the number of $k$-flats in $\ag$ contained in $U$ is $q^{m-k}{m\b k}$.
		\item[\rmt] If $m\le k\le n$, the number of $k$-flats in $\ag$ containing $U$ is ${n-m\b k-m}$.
	\end{itemize}
\end{lem}

For a flat $U$, let $U'$ denote the subspace of $\ff^n$ relative to $U$.
\begin{lem}\label{af-weishu}\rm{(\cite{aff-add,JS})} Let $F_1=F_1'+f_1$ and $F_2=F_2'+f_2$ be two flats in $\ag$. The following hold.
	\begin{itemize}
		\item[\rmo] $F_1\cap F_2\neq\emptyset$ if and only if $f_1-f_2\in F_1'+F_2'$.
		\item[\rmt] If $F_1\cap F_2\neq\emptyset$, then $F_1\cap F_2=F_1'\cap F_2'+x$, where $x\in F_1\cap F_2$.
		\item[\rmth] $F_1\vee F_2=F_1'+F_2'+\langle f_2-f_1\rangle+f_1$. Moreover,
		\begin{equation*}
			\dim(F_1\vee F_2)=\left\{
			\begin{array}{ll}
				\dim F_1+\dim F_2-\dim(F_1\cap F_2), &\text{if}\ F_1\cap F_2\neq\emptyset,\\
				\dim F_1+\dim F_2-\dim(F_1'\cap F_2')+1, &\text{if}\ F_1\cap F_2=\emptyset.
			\end{array}
			\right.
		\end{equation*}
	\end{itemize}
\end{lem}

For $A,B\in\ag$ and $\mf\subset\kn$, we say $A$ \emph{$t$-intersects} $B$ if $\dim(A\cap B)\ge t$, and write
$$\quad\mf_A=\{F\in\mf: A\subset F\}.$$
For $S\in\mo(s,n)$, $T\in\mo(t,S)$ and $j\in\{t,t+1,\dots,s\}$, write
\begin{equation*}\label{af-xla}
	\begin{aligned}
		\ml_j(S,T;k)&=\{(I,F)\in\mo(j,n)\times\mo(k,n): T\subset I\subset S, I\subset F\},\\
		\ma_j(S,T;k)&=\{F\in\mo(k,n): T\subset F, \dim(F\cap S)=j\}
	\end{aligned}
\end{equation*}
and
$$a_0(n,k,s,t)={s-t\b1}{n-t-1\b k-t-1}-q{s-t\b2}{n-t-2\b k-t-2}.$$
\begin{lem}\label{af-suanliangci-xiajie}
Let $n$, $k$, $s$ and $t$ be positive integers with $n>k,s$ and $k,s\ge t+1$. Suppose $S\in\mo(s,n)$ and $T\in\mo(t,S)$. Then
$$a_0(n,k,s,t)\le|\{F\in\mo(k,n): T\subset F, \dim(F\cap S)\ge t+1\}|.$$
\end{lem}
\begin{proof}
For each $j\in\{t+1,\dots,s\}$, using Lemma \ref{af-jishu},  by double counting $|\ml_j(S,T;k)|$, we have
\begin{equation}\label{af-suanliangci-formu}
	|\ml_j(S,T;k)|={s-t\b j-t}{n-j\b k-j}=\sum_{i=j}^s{i-t\b j-t}|\ma_i(S,T;k)|.
\end{equation}
Then
\begin{equation*}
	\begin{aligned}
		a_0(n,k,s,t)=&\ |\ml_{t+1}(S,T;k)|-q|\ml_{t+2}(S,T;k)|\\
		=&\sum_{i=t+1}^s{i-t\b1}|\ma_i(S,T;k)|-q\sum_{i=t+2}^s{i-t\b2}|\ma_i(S,T;k)|\\
		=&\ |\ma_{t+1}(S,T;k)|+|\ma_{t+2}(S,T;k)|+\sum_{i=3}^s\left({i-t\b1}-q{i-t\b2}\right)|\ma_i(S,T;k)|.
	\end{aligned}
\end{equation*}
Note that ${i-t\b1}<q{i-t\b2}$ for $i\ge t+3$. We have
\begin{equation*}
	\begin{aligned}
		a_0(n,k,s,t)&\le|\ma_{t+1}(S,T;k)|+|\ma_{t+2}(S,T;k)|\\
		&=|\{F\in\mo(k,n):T\subset F, \dim(F\cap S)\in\{t+1,t+2\}\}|\\
		&\le|\{F\in\mo(k,n): T\subset F, \dim(F\cap S)\ge t+1\}|,
	\end{aligned}
\end{equation*}
as desired.
\end{proof}

Let $\mf\subset\kn$. For a flat $T$ in $\ag$, if $\dim(T\cap F)\ge t$ holds for each $F\in\mf$, we say $T$ is a \emph{$t$-cover} of $\mf$. Define
$$\tau_t(\mf):=\min\{\dim T: T\ \text{is a}\ t\text{-cover of}\ \mf\}.$$

\begin{lem}\label{af-fangsuo}
Let $n$, $k$, $s$ and $t$ be positive integers with $n>k\ge s\ge t$. Suppose $\mf\subset\mo(k,n)$, $X$ is a $t$-cover of $\mf$  with dimension $x$ and $S\in\mo(s,n)$. If $X$ does not $t$-intersect $S$, then there exists a flat $R$ in $\ag$ such that $S\subsetneq R$ and
\begin{equation}\label{af-fangsuo-formu}
	|\mf_S|\le{x-t+1\b1}^{\dim R-\dim S}|\mf_R|,
\end{equation}
Moreover, if $n\ge x+k-t+1$, then
$$|\mf_S|\le{x-t+1\b1}{n-s-1\b k-s-1}.$$
\end{lem}
\begin{proof}
	W.l.o.g., assume that $\mf_S\neq\emptyset$.

Let $U=X\vee S$, $u=\dim(S'\cap X')$ and 
$$\mr=\{R\in\mo(r,n): S\subset R\subset U\},$$
where
\[r=
\begin{cases}
	s+t-u,&X\cap S\neq\emptyset;\\
	s+t-u+1,&X\cap S=\emptyset\ \mbox{and}\ u<t;\\
	s+1,&X\cap S=\emptyset\ \mbox{and}\ u\ge t.
\end{cases}
\]
For each $F\in\mf_S$, we have $\dim(X\cap F)\ge t$ and 
$$(S\vee(F\cap X))\subset F\cap U,\quad\dim(S'\cap(F\cap X)')=\dim(S'\cap X')=u.$$

Suppose $X\cap S\neq\emptyset$. We have $S\cap(F\cap X)\neq\emptyset$. By Lemma \ref{af-weishu}, we get
$$\dim(F\cap U)\ge\dim(S\vee(F\cap X))\ge s+t-\dim(S'\cap(F\cap X)')=s+t-u.$$
Together with $\dim U=x+s-u$, we obtain
$$|\mr|={x-u\b t-u}\le{x-t+1\b1}^{t-u}={x-t+1\b1}^{r-s},\quad\mf_S\subset\bigcup_{R\in\mr}\mf_R.$$
Let $R\in\mr$ such that $|\mf_R|$ is the largest. Then \eqref{af-fangsuo-formu} follows.

Similarly, when $X\cap S=\emptyset$ and $u<t$, \eqref{af-fangsuo-formu} also follows.

Now assume that $X\cap S=\emptyset$ and $u\ge t$. By $S\subset F$, we have 
$$\dim(F\cap X)\ge\dim(S'\cap X')=u.$$ 
Thus
$$\dim(F\cap U)\ge\dim(S\vee(F\cap X))\ge s+u-u+1=s+1.$$
Together with $\dim U=x+s-u+1$, we obtain
$$|\mr|={x-u+1\b 1}\le{x-t+1\b1},\quad\mf_S\subset\bigcup_{R\in\mr}\mf_R.$$
Let $R\in\mr$ such that $|\mf_R|$ is the largest. Then \eqref{af-fangsuo-formu} follows.

If $k>s$, for $v\in\{s+1,\dots,k-1\}$, by Lemma \ref{gaosifangsuo} and $n\ge x+k-t+1$, we have
\begin{equation}\label{add}
	\dfrac{{x-t+1\b1}^{v+1-s}{n-v-1\b k-v-1}}{{x-t+1\b1}^{v-s}{n-v\b k-v}}=\dfrac{(q^{x-t+1}-1)(q^{k-v}-1)}{(q-1)(q^{n-v}-1)}\le q^{x+k-t+1-n}\le1.
\end{equation}
Then
$$|\mf_S|\le{x-t+1\b1}^{r-s}{n-r\b k-r}\le{x-t+1\b1}{n-s-1\b k-s-1},$$
as desired.
\end{proof}

\begin{lem}\label{af-yibanshangjie}
Let $n$, $k_1$, $k_2$ and $t$ be positive integers with $n\ge k_1+k_2-t+1$  and $k_1,k_2\ge t$. Suppose that $\mf\subset\mo(k_1,n)$ and $\mg\subset\mo(k_2,n)$ are cross $t$-intersecting. Then
	$$|\mf|\le q^{\tau_t(\mf)-t}{\tau_t(\mf)\b t}{k_2-t+1\b1}^{\tau_t(\mg)-t}{n-\tau_t(\mg)\b k_1-\tau_t(\mg)}.$$
\end{lem}
\begin{proof}
Let $S$ be a $t$-cover of $\mf$ with dimension $\tau_t(\mf)$. By
\begin{equation}\label{af-yibanshangjie-f1}
	\mf\subset\bigcup_{W\in\mo(t,S)}\mf_W,
\end{equation}
we have
$$|\mf|\le q^{\tau_t(\mf)-t}{\tau_t(\mf)\b t}{n-t\b k_1-t}.$$
Therefore, when $\tau_t(\mg)=t$, the desired result holds. In the following, we may assume that $\tau_t(\mg)>t$.

Let $W_1\in\mo(t,S)$ with $\mf_{W_1}\neq\emptyset$. 
We first give an upper bound for $|\mf_{W_1}|$.
Since $\tau_t(\mg)>t$, there exists $G_1\in\mg$ such that $W_1$ does not $t$-intersect $G_1$. 
Note that $G_1$ is a $t$-cover of $\mf$. By Lemma \ref{af-fangsuo}, there exists a flat $W_2$ such that $\dim W_2>\dim W_1$ and
$$|\mf_{W_1}|\le{k_2-t+1\b1}^{\dim W_2-\dim W_1}|\mf_{W_2}|.$$
By $|\mf_{W_1}|>0$, we have $|\mf_{W_2}|>0$, which implies that $\dim W_2\le k_1$. 
If $\dim W_2<\tau_t(\mg)$, then there exists $G_2\in\mg$ such that $W_2$ does not $t$-intersect $G_2$.
Using Lemma \ref{af-fangsuo} repeatedly, there exist some flats $W_1,W_2,\dots,W_u$ such that $\dim W_{u-1}<\tau_t(\mg)\le\dim W_u\le k_1$ and
$$|\mf_{W_i}|\le{k_2-t+1\b1}^{\dim W_{i+1}-\dim W_i}|\mf_{W_{i+1}}|$$
 for each $i\in\{1,\dots,u-1\}$. 
Then
$$|\mf_{W_1}|\le{k_2-t+1\b1}^{\dim W_u-t}|\mf_{W_u}|\le{k_2-t+1\b1}^{\dim W_u-t}{n-\dim W_u\b k_1-\dim W_u}.$$
Together with $n\ge k_1+k_2-t+1$, \eqref{add} and $\dim W_u\ge\tau_t(\mg)$, we get
\begin{equation*}
	\begin{aligned}
		|\mf_{W_1}|\le{k_2-t+1\b1}^{\dim W_u-t}{n-\dim W_u\b k_1-\dim W_u}\le{k_2-t+1\b1}^{\tau_t(\mg)-t}{n-\tau_t(\mg)\b k_1-\tau_t(\mg)}.
	\end{aligned}
\end{equation*}
Then by (\ref{af-yibanshangjie-f1}), we obtain
$$|\mf|\le\sum_{W\in\mo(t,S)}|\mf_W|\le q^{\tau_t(\mf)-t}{\tau_t(\mf)\b t}{k_2-t+1\b1}^{\tau_t(\mg)-t}{n-\tau_t(\mg)\b k_1-\tau_t(\mg)},$$
as desired.
\end{proof}

\section{Proof of Theorem \ref{trivial}}

To prove Theorem \ref{trivial}, we need the following two lemmas.

\begin{lem}\label{af-dijian}
	Let $n$, $b$, $c$ and $t$ be positive integers with $n\ge b+c+3$ and $b\ge t+1$, $c\ge t$. For $x\in\{t,\dots,b\}$, write
	$$h_{b,c}(x)=q^{x-t}{x\b t}{c-t+1\b1}^{x-t}{n-x\b b-x},$$
	Then $h_{b,c}(x)$ is decreasing with respect to $x$.
\end{lem}
\begin{proof}
By Lemma \ref{gaosifangsuo} and $n\ge b+c+3$, for each $x\in\{t,\dots,b-1\}$, we have
$$\dfrac{h_{b,c}(x+1)}{h_{b,c}(x)}=\dfrac{q(q^{x+1}-1)(q^{c-t+1}-1)(q^{b-x}-1)}{(q^{x-t+1}-1)(q-1)(q^{n-x}-1)}<q^{b+c+3-n}\le1.$$
Then the desired result follows.
\end{proof}

\begin{lem}\label{af-cover1}
	Let $n$, $k_1$, $k_2$ and $t$ be positive integers with $n\ge k_1+k_2+1$ and $k_1,k_2\ge t$. Suppose $\mf_1\subset\mo(k_1,n)$ and $\mf_2\subset\mo(k_2,n)$ are maximal cross $t$-intersecting families. For each $i\in\{1,2\}$, let $\ms_i$ denote the set of all $t$-covers of $\mf_i$ with dimension $\tau_t(\mf_i)$. Then $\ms_1$ and $\ms_2$ are cross $t$-intersecting.
\end{lem}
\begin{proof}
	Suppose $S_1=S_1'+s_1\in\ms_1$ and $S_2=S_2'+s_2\in\ms_2$. It is sufficient to show that $\dim(S_1\cap S_2)\ge t$.
	
	Since $n\ge k_1+k_2+1$, there exist $k_1$-flat $F_1$ and $k_2$-flat $F_2$ such that $S_1\subset F_2$ and $S_2\subset F_1$. By the maximality of $\mf_1$ and $\mf_2$, we have $F_1\in\mf_1$ and $F_2\in\mf_2$.
	
	If $S_1\cap S_2=\emptyset$, then $s_1-s_2\not\in S_1'+S_2'$. Since $n\ge k_1+k_2+1$, we may assume that $(F_1'+F_2')\cap\langle s_1-s_2\rangle=\{0\}$. Together with $F_1=F_1'+s_2$ and $F_2=F_2'+s_1$, we have $F_1\cap F_2=\emptyset$, a contradiction. Thus $S_1\cap S_2\neq\emptyset$.
	
	By $n\ge k_1+k_2+1$, we may assume that $F_1'\cap F_2'=S_1'\cap S_2'$. Together with $S_1\cap S_2\neq\emptyset$, we get
	$$\dim(S_1\cap S_2)=\dim(S_1'\cap S_2')=\dim(F_1'\cap F_2')=\dim(F_1\cap F_2)\ge t,$$
	as desired.
\end{proof}

\begin{proof}[\bf Proof of Theorem \ref{trivial}]
	
	Suppose that $\mf_1\subset\mo(k_1,n)$ and $\mf_2\subset\mo(k_2,n)$ are cross $t$-intersecting. If $\tau_t(\mf_1)=\tau_t(\mf_2)=t$, let $T_1$ and $T_2$ be $t$-covers of $\mf_1$ and $\mf_2$ with dimension $t$, respectively. By Lemma \ref{af-cover1}, we have $T_1=T_2:=T$.
	Then
	$$|\mf_1|\le{n-t\b k_1-t},\quad|\mf_2|\le{n-t\b k_2-t},$$
	and two equalities hold at the same time if and only if $\mf_i=\{F\in\mo(k_i,n): T\subset F\}$ for each $i\in\{1,2\}$.
	
	To finish our proof, it is sufficient to show that if $(\tau_t(\mf_1),\tau_t(\mf_2))\neq(t,t)$, then
	\begin{equation}\label{af-f1f2<nn}
		|\mf_1||\mf_2|<{n-t\b k_1-t}{n-t\b k_2-t}.
	\end{equation}
	By Lemma \ref{af-yibanshangjie} and $n\ge k_1+k_2-t+1$, we have
	\begin{equation*}
		\begin{aligned}
			|\mf_1||\mf_2|\le&\left(q^{\tau_t(\mf_1)-t}{\tau_t(\mf_1)\b t}{k_1-t+1\b1}^{\tau_t(\mf_1)-t}{n-\tau_t(\mf_1)\b k_2-\tau_t(\mf_1)}\right)\\
			&\cdot\left(q^{\tau_t(\mf_2)-t}{\tau_t(\mf_2)\b t}{k_2-t+1\b1}^{\tau_t(\mf_2)-t}{n-\tau_t(\mf_2)\b k_1-\tau_t(\mf_2)}\right).
		\end{aligned}
	\end{equation*}
	Note that $t\le\tau_t(\mf_1)\le k_2$ and $t\le\tau_t(\mf_2)\le k_1$.
	Then \eqref{af-f1f2<nn} follows from $n\ge k_1+k_2+3$, Lemma \ref{af-dijian} and $(\tau_t(\mf_1),\tau_t(\mf_2))\neq(t,t)$.
\end{proof}

Based on Theorem \ref{trivial},  we obtain a more general theorem.

\begin{thm}\label{trivial2}
	Let $d$, $n$, $t$, $k_1$,\dots, $k_d$ be positive integers with $d\ge2$, $k_1\ge k_2\ge\cdots\ge k_d\ge t$ and $n\ge k_1+k_2+3$. If $\mf_1\subset\mo(k_1,n)$,\dots, $\mf_d\subset\mo(k_d,n)$ satisfy that $\dim(F_1\cap\cdots\cap F_d)\ge t$ 
	for any $F_i\in\mf_i$, $i=1,\dots,d$. If $\prod_{i=1}^d|\mf_i|$ reaches to the maximum value, then there exists a $t$-flat $T$ such that $\mf_i=\{F\in\mo(k_i,n): T\subset F\}$ for each $i\in\{1,\dots,d\}$.
\end{thm}
\begin{proof}
	For distinct $i,j\in\{1,\dots,d\}$, $\mf_i$ and $\mf_j$ are cross $t$-intersecting families. Then by Theorem \ref{trivial}, we have
	$$|\mf_i||\mf_j|\le{n-t\b k_i-t}{n-t\b k_j-t}.$$
	Thus
	\begin{equation}\label{af-cross-prod}
		\begin{aligned}
			\left(\prod_{s=1}^d|\mf_s|\right)^{d-1}=\prod_{1\le i<j\le d}|\mf_i||\mf_j|\le\prod_{1\le i<j\le d}{n-t\b k_i-t}{n-t\b k_j-t}=\left(\prod_{s=1}^d{n-t\b k_s-t}\right)^{d-1},
		\end{aligned}
	\end{equation}
	and equality holds if and only if $|\mf_i||\mf_j|={n-t\b k_i-t}{n-t\b k_j-t}$ for any $i,j\in\{1,\dots,d\}$.
	
	Note that the product the sizes of families $\{F\in\mo(k_1,n):S\subset F\}$, $i=1,\dots,d$, reaches to the upper bound of \eqref{af-cross-prod}, where $S$ is a $t$-flat.
	Therefore, by assumption and Theorem \ref{trivial}, for distinct $i,j\in\{1,\dots,d\}$, there exists a $t$-flat $T_{i,j}$ such that
	$$\mf_i=\{F\in\mo(k_i,n): T_{i,j}\subset F\},\quad\mf_j=\{F\in\mo(k_j,n): T_{i,j}\subset F\}.$$
	If there exists $j'\in\{1,2,\dots,d\}$, such that $T_{i,j'}\neq T_{i,j}$, then 
	$$\mf_i\subset\{F\in\mo(k_i,n): T_{i,j}\vee T_{i,j'}\subset F\}.$$
	Together with Lemma \ref{af-jishu} and $n\ge k_1+k_2+3$, $k_1\ge k_i$, $\dim(T_{i,j}\vee T_{i,j'})\ge t+1$, we get
	$${n-t-1\b k_i-t-1}<{n-t\b k_i-t}=|\mf_i|\le{n-t-1\b k_i-t-1},$$
	a contradiction.  Therefore, there exists a $t$-flat $T$ such that $T_{i,j}=T$ for any distinct $i,j\in\{1,\dots,d\}$. Then the desired result follows.
\end{proof}

\section{Proof of Theorem \ref{non-trivial}}

Suppose $M\in\mm(k_2+1,n)$,  $S\in\mm(t+1,n)$ and $T\in\mm(t,M)$. Let 	$\mma_1(M,T;k_1,t)$, 
$\mma_2(M,T;k_2,t)$, 
$\mma_3(S;k_1)$ and
$\mma_4(S;k_2,t)$ are families defined in \eqref{c1-c4}. By \cite[Theorem 1.21]{JS}, $|\mma_1(M,T;k_1,t)|\cdot|\mma_2(M,T;k_2,t)|$ and $|\mma_3(S;k_1)|\cdot|\mma_4(S;k_2,t)|$ are independent on the choice of $M$, $T$ and $S$. Write
\begin{equation*}
	\begin{aligned}
		a_1(n,k_1,k_2,t)&=|\mma_1(M,T;k_1,t)|\cdot|\mma_2(M,T;k_2,t)|,\\
		a_2(n,k_1,k_2,t)&=|\mma_3(S;k_1)|\cdot|\mma_4(S;k_2,t)|.
	\end{aligned}
\end{equation*}

By Lemma \ref{af-jishu}, we have
\begin{equation}\label{cf3cf4-size}
	\dfrac{a_2(n,k_1,k_2,t)}{{n-t-1\b k_1-t-1}}=q{t+1\b1}{n-t\b k_2-t}-\left(q{t+1\b1}-1\right){n-t-1\b k_2-t-1}
\end{equation}
and
\begin{equation}\label{af-weiba}
	\begin{aligned}
		&|\{F\in\mo(k_2,M): \dim(F\cap T)=t-1\}|\\
		=&\sum_{U\in\mo(t-1,T)}|\{F\in\mo(k_2,M): U\subset F\}\bs\{F\in\mo(k_2,M): T\subset F\}|\\
		=&\ q{t\b1}\left({k_2+1-(t-1)\b1}-{k_2-t+1\b1}\right)\\
		=&\ q^{k_2-t+2}{t\b1}.
	\end{aligned}
\end{equation}
We first prove some inequalities for $a_1(n,k_1,k_2,t)$  and $a_2(n,k_1,k_2,t)$.

\begin{lem}\label{af-cf1-cf2}
	Let $n$, $k_1$, $k_2$ and $t$ be positive integers with $k_1\ge k_2\ge t+1$ and $n\ge k_1+k_2+3$.
	\begin{itemize}
		\item[\rmo] If $k_2\ge2t+1$, then $a_1(n,k_1,k_2,t)>a_2(n,k_1,k_2,t)$.
		\item[\rmt] If $k_2<2t+1$, then $a_1(n,k_1,k_2,t)<a_2(n,k_1,k_2,t)$.
	\end{itemize}
\end{lem}
\begin{proof}
Write
$$a_3(n,k_1,k_2,t)=a_2(n,k_1,k_2,t)-a_1(n,k_1,k_2,t).$$

(1) Assume that $k_2\ge2t+1$. For $a\in\{t+1,\dots,k_2\}$, by Lemma \ref{gaosifangsuo} and $n\ge k_1+k_2-t$, we have
\begin{align*}
	\dfrac{a_0(n,k_1,a+1,t)-a_0(n,k_1,a,t)}{q^{a-t}{n-t-1\b k_1-t-1}}&=1-\dfrac{q^{k_1-t-1}-1}{q^{n-t-1}-1}{a-t\b1}\ge1-q^{k_1+k_2-t-n}\ge0.
\end{align*}
 Let $M\in\mm(k_2+1,n)$ and $T\in\mm(t,M)$. 
Together with Lemma \ref{af-suanliangci-xiajie}, \eqref{cf3cf4-size} and $k_2\ge2t+1$, $n\ge k_1+k_2+3\ge k_1+2t+3$, we get
\begin{align*}
	\dfrac{a_3(n,k_1,k_2,t)}{{n-t-1\b k_1-t-1}{n-t\b k_2-t}}&<\dfrac{a_2(n,k_1,k_2,t)}{{n-t-1\b k_1-t-1}{n-t\b k_2-t}}-\dfrac{|\mma_1(M,T;k_1,t)|}{{n-t-1\b k_1-t-1}}\\
	&\le q{t+1\b1}-\dfrac{a_0(n,k_1,k_2+1,t)}{{n-t-1\b k_1-t-1}}\\
	&\le q{t+1\b1}-\dfrac{a_0(n,k_1,2t+2,t)}{{n-t-1\b k_1-t-1}}\\
	&=\dfrac{q(q^{k_1-t-1}-1)}{q^{n-t-1}-1}{t+2\b2}-1\\
	&<q^{k_1+2t+3-n}-1\\
	&\le0,
\end{align*}
which implies that $a_1(n,k_1,k_2,t)>a_2(n,k_1,k_2,t)$.

(2) Now assume that $k_2\le2t$. By Lemma \ref{af-jishu} and \eqref{af-weiba}, we obtain
\begin{align*}
	a_1(n,k_1,k_2,t)&\le{k_2-t+1\b1}{n-t-1\b k_1-t-1}\left({n-t\b k_2-t}+q^{k_2-t+2}{t\b1}\right)\\
	&\le{t+1\b1}{n-t-1\b k_1-t-1}\left({n-t\b k_2-t}+q^{k_2-t+2}{t\b1}\right).
\end{align*}
Together with $n\ge k_1+k_2+3$, $k_1\ge k_2$, \eqref{cf3cf4-size} and Lemma \ref{gaosifangsuo}, we get
\begin{align*}
	\dfrac{a_3(n,k_1,k_2,t)}{{n-t-1\b k_1-t-1}{n-t\b k_2-t}}&\ge q{t+1\b1}-\left(q{t+1\b1}-1\right)\dfrac{q^{k_2-t}-1}{q^{n-t}-1}-{t+1\b1}\left(1+\dfrac{q^{k_2-t+2}{t\b1}}{{n-t\b k_2-t}}\right)\\
	&\ge(q-1){t+1\b1}-q^{k_2+t+2-n}-\dfrac{q^{k_2+t+3}}{q^{(k_2-t)(n-k_2)}}\\
	&\ge q^{t+1}-1-q^{k_2+t+2-n}-q^{2k_2+t+3-n}\\
	&\ge q^{t+1}-1-q^{-k_1+t-1}-q^{-k_1+k_2+t}\\
	&\ge q^{t+1}-1-q^{-2}-q^{t}\\
	&>0,
\end{align*}
which implies that $a_1(n,k_1,k_2,t)<a_2(n,k_1,k_2,t)$.
\end{proof}

\begin{lem}\label{af-cf-cf'}
		Let $n$, $k_1$, $k_2$ and $t$ be positive integers with $n\ge k_1+k_2+t+3$ and $k_1>k_2\ge t+1$. Then the following hold.
	\begin{itemize}
		\item[\rmo] $a_1(n,k_1,k_2,t)>a_1(n,k_2,k_1,t)$.
		\item[\rmt] $a_2(n,k_1,k_2,t)>a_2(n,k_2,k_1,t)$.
	\end{itemize}
\end{lem}
\begin{proof}

(1) Suppose $M\in\mm(k_2+1,n)$ and $T\in\mm(t,M)$. By Lemmas \ref{af-QY} and \ref{af-weishu}, we have
\begin{equation}\label{af-cf1-size}
	\begin{aligned}
		|\mma_1(M,T;k_1,t)|&=\left|\left\{K\in{\ff^n\b k_1}: T'\subset K, \dim(K\cap M')\ge t+1\right\}\right|\\
		&={n-t\b k_1-t}-q^{(k_2-t+1)(k_1-t)}{n-k_2-1\b k_1-t}.
	\end{aligned}
\end{equation}
Together with \eqref{af-weiba}, we get
$$a_1(n,k_1,k_2,t)=\left({n-t\b k_1-t}-q^{(k_2-t+1)(k_1-t)}{n-k_2-1\b k_1-t}\right)\left({n-t\b k_2-t}+q^{k_2-t+2}{t\b1}\right).$$
Then from Lemma \ref{af-suanliangci-xiajie}, we obtain
\begin{align*}
	&\ a_1(n,k_1,k_2,t)-a_1(n,k_2,k_1,t)\\
	=&\left(q^{(k_1-t+1)(k_2-t)}{n-k_1-1\b k_2-t}{n-t\b k_1-t}-q^{(k_2-t+1)(k_1-t)}{n-k_2-1\b k_1-t}{n-t\b k_2-t}\right)\\
	&+q^{k_2-t+2}{t\b1}\left({n-t\b k_1-t}-q^{(k_2-t+1)(k_1-t)}{n-k_2-1\b k_1-t}\right)\\
	&-q^{k_1-t+2}{t\b1}\left({n-t\b k_2-t}-q^{(k_1-t+1)(k_2-t)}{n-k_1-1\b k_2-t}\right)\\
	\ge&\left(q^{(k_1-t+1)(k_2-t)}{n-k_1-1\b k_2-t}{n-t\b k_1-t}-q^{(k_2-t+1)(k_1-t)}{n-k_2-1\b k_1-t}{n-t\b k_2-t}\right)\\
	&+q^{k_2-t+2}{t\b1}\left(a_0(n,k_1,k_2+1,t)-q^{k_1-k_2}{k_1-t+1\b1}{n-t-1\b k_2-t-1}\right)
\end{align*}
Since~$n\ge k_1+k_2+t+3$, by the proof of \cite[Lemma 5.6]{CAO}, we have
\begin{align*}
	&q^{(k_1-t+1)(k_2-t)}{n-k_1-1\b k_2-t}{n-t\b k_1-t}-q^{(k_2-t+1)(k_1-t)}{n-k_2-1\b k_1-t}{n-t\b k_2-t}>0,\\
	&a_0(n,k_1,k_2+1,t)-q^{k_1-k_2}{k_1-t+1\b1}{n-t-1\b k_2-t-1}>0.
\end{align*}
 Then $a_1(n,k_1,k_2,t)>a_1(n,k_2,k_1,t)$.

(2) By $k_1>k_2$ and \eqref{cf3cf4-size}, we have
$$\dfrac{a_2(n,k_1,k_2,t)-a_2(n,k_2,k_1,t)}{q{t+1\b1}{n-t-1\b k_1-t-1}{n-t-1\b k_2-t-1}}=\dfrac{q^{n-t}-1}{q^{k_2-t}-1}-\dfrac{q^{n-t}-1}{q^{k_1-t}-1}>0.$$
Then the desired result follows.
\end{proof}

To present the proof of Theorem \ref{non-trivial} concisely, we need the following two lemmas.
\begin{lem}\label{af-upper-other}
	Let $n$, $k_1$, $k_2$ and $t$ be positive integers with $k_1,k_2\ge t+1$ and $n\ge k_1+k_2+t+7$. Suppose $\mf_1\subset\mo(k_1,n)$ and $\mf_2\subset\mo(k_2,n)$ are non-trivial cross $t$-intersecting with $\tau_t(\mf_2)\ge\tau_t(\mf_1)$ and $(\tau_t(\mf_1),\tau_t(\mf_2))\neq(t,t+1)$. Then $|\mf_1||\mf_2|<a_1(n,k_1,k_2,t)$.
\end{lem}
\begin{proof}
Suppose $\tau_t(\mf_1)=t$. We have $\tau_t(\mf_2)\ge t+2$. By Lemmas \ref{af-yibanshangjie} and \ref{af-dijian}, we obtain
\begin{align*}
	|\mf_1||\mf_2|&\le\left(q^{\tau_t(\mf_2)-t}{\tau_t(\mf_2)\b t}{k_2-t+1\b1}^{\tau_t(\mf_2)-t}{n-\tau_t(\mf_2)\b k_1-\tau_t(\mf_2)}\right){n-t\b k_2-t}\\
	&\le\left(q^2{t+2\b 2}{k_2-t+1\b1}^{2}{n-t-2\b k_1-t-2}\right){n-t\b k_2-t}.
\end{align*}
By Lemma \ref{gaosifangsuo} and $n\ge k_1+k_2+t+7$, we have
\begin{align*}
	&\left({a_0(n,k_1,k_2+1,t)-q^2{t+2\b 2}{k_2-t+1\b1}^{2}{n-t-2\b k_1-t-2}}\right){n-t-2\b k_1-t-2}^{-1}\\
	=&\ \dfrac{q^{n-t-1}-1}{q^{k_1-t-1}-1}{k_2-t+1\b1}-q{k_2-t+1\b2}-q^2{t+2\b 2}{k_2-t+1\b1}^{2}\\
	>&\ q^{n-k_1+k_2-t}-q^{2k_2-2t+1}-q^{2k_2+6}\\
	>&\ q^{n-k_1+k_2-t}-q^{2k_2+7}\\
	\ge&\ 0.
\end{align*}
Together with Lemma \ref{af-suanliangci-xiajie}, we get
$$|\mf_1||\mf_2|<a_0(n,k_1,k_2+1,t){n-t\b k_2-t}\le a_1(n,k_1,k_2,t).$$

Now assume that $\tau_t(\mf_1)\ge t+1$. Note that $\tau_t(\mf_2)\ge t+1$.
By Lemmas \ref{af-yibanshangjie} and \ref{af-dijian}, we have
\begin{align*}
	|\mf_1||\mf_2|&\le q^2{t+1\b1}^2{k_1-t+1\b1}{k_2-t+1\b1}{n-t-1\b k_1-t-1}{n-t-1\b k_2-t-1}.
\end{align*}
Then from Lemmas \ref{gaosifangsuo}, \ref{af-suanliangci-xiajie} and $n\ge k_1+k_2+t+7$, we obtain
\begin{align*}
	&\left(a_1(n,k_1,k_2,t)-|\mf_1||\mf_2|\right){k_2-t+1\b1}^{-1}{n-t-1\b k_1-t-1}^{-1}{n-t\b k_2-t}^{-1}\\
	\ge&\ 1-\dfrac{q(q^{k_1-t-1}-1)(q^{k_2-t}-1)}{(q^{n-t-1}-1)(q^2-1)}-\dfrac{q^2(q^{k_2-t}-1)}{q^{n-t}-1}{t+1\b1}^2{k_1-t+1\b1}\\
	>&\ 1-q^{k_1-n}\cdot\dfrac{q^{k_2-t}-1}{q-1}-q^{k_1+k_2+t+5-n}\\
	\ge&\ 1-q^{-2t-7}-q^{-2}\\
	>&\ 0,
\end{align*}
as desired.
\end{proof}

\begin{lem}\label{af-upper-tt+1}
	Let $n$, $k_1$, $k_2$ and $t$ be positive integers with $k_1,k_2\ge t+1$ and $n\ge k_1+k_2+t+7$. Suppose $\mf_1\subset\mo(k_1,n)$ and $\mf_2\subset\mo(k_2,n)$ are maximal non-trivial cross $t$-intersecting families with $(\tau_t(\mf_1),\tau_t(\mf_2))=(t,t+1)$. Then one of the following holds.
	\begin{itemize}
		\item[\rmo] $\mf_1=\mma_1(M,T;k_1,t)$ and $\mf_2=\mma_2(M,T;k_2,t)$ for some $M\in\mo(k_2+1,n)$ and $T\in\mo(t,M)$.
		\item[\rmt] $\mf_1=\mma_3(S;k_1)$ and $\mf_2=\mma_4(S;k_2,t)$ for some $S\in\mo(t+1,n)$.
		\item[\rmth] $|\mf_1||\mf_2|<a_1(n,k_1,k_2,t).$
	\end{itemize}
\end{lem}
\begin{proof}
	Denote the set of all $t$-covers of $\mf_2$ with dimension $t+1$ by $\ms$. Let $T\in\mm(t,n)$ be a $t$-cover of $\mf_1$,  $M=\bigvee\limits_{S\in\ms}T$ and $m=\dim M$. By Lemma \ref{af-cover1}, each member of $\ms$ contains $T$.
	
	Let $S\in\ms$ and $F\in\mf_2\bs(\mf_2)_T$. By $\dim(S\cap F)\ge t$ and $T\not\subset F$, we have $\dim(S\cap F)=t$. Note that $T\vee F\subset S\vee F$. Together with Lemma \ref{af-weishu}, we get
	$$k_2+1\le\dim(T\vee F)\le\dim(S\vee F)=k_2+1.$$
	Hence $T\vee F=S\vee F$, which implies that $T\vee F=M\vee F$ and $\dim(M\vee F)=k_2+1$. Observe that $F\cap M\neq\emptyset$. By Lemma \ref{af-weishu}, we have $\dim(F\cap M)=m-1$. From $S\subset M$ and $\dim(M\vee F)=k_2+1$, we obtain $t+1\le m\le k_2+1$.
	
	Since $\tau_t(\mf_2)=t+1$, there exists $F_{2,1}\in\mf_2$ such that $T\not\subset F_{2,1}$. Write $$H:=T\vee F_{2,1}.$$ 
	Observe that
	$$S=T\vee(S\cap F_{2,1})\subset H,\quad \dim H=k_2+1.$$
	 For each $F\in\mf_1$, since $T\subset F$, $T\not\subset F_{2,1}$ and $\dim(F\cap F_{2,1})\ge t$, we have $\dim(F\cap H)\ge t+1$. Therefore
	\begin{equation}\label{af-f1-str}
		\mf_1\subset\{F\in\mm(k_1,n): T\subset F, \dim(F\cap H)\ge t+1\}.
	\end{equation}

\medskip
\noindent{\bf Case~1. $m=k_2+1$.}
\medskip

In this case, we have $M=H$. Notice that
\begin{align*}
	\mf_2\subset\{F\in\mo(t,k_2): T\subset F\}\cup\{F\in\mo(k_2,M): T\not\subset F\}.
\end{align*}

If $T\cap G\neq\emptyset$ for each $G\in\mf_2\bs(\mf_2)_T$, by Lemma \ref{af-weishu}, we have $\dim(G\cap T)=t-1$. Thus $\mf_2\subset\mma_2(M,T;k_2)$. Then (1) follows from \eqref{af-f1-str} and the maximality of $\mf_1$ and $\mf_2$.

Now assume that there exists $A=A'+a\in\mf_2$ such that $A\cap T=\emptyset$. By Lemma \ref{af-weishu} and $A\subset M$, we have $T'\subset A'$. Let $b\in T$ and $B:=B'+b$ be a $(t+1)$-subflat of $M$ with $B'\in{A'\b t+1}$ and $T\subset B$. Since $A\cap T=\emptyset$, we have $b-a\not\in A'$, which implies that $B\cap A=\emptyset$.
Write
\begin{align*}
	&\mr_1=\{F\in\mo(k_1,n): T\subset F,\dim(F\cap A)\ge t\},\\
	&\mr_2=\{F\in\mo(k_1,n): T\subset F, F\cap M=B\}.
\end{align*}
Observe that $\mf_1\subset\mr_1$, $\mr_1\cap\mr_2=\emptyset$  and $\mr_1\cup\mr_2\subset\mma_1(M,T;k_1,t)$.
By Lemmas \ref{gaosifangsuo}--\ref{af-weishu}, we have
\begin{align*}
	|\mr_2|&=\left|\left\{K\in{\ff^n\b k_1}: T'\subset K, K\cap M'=B'\right\}\right|\\
	&=\ q^{(k_2-t)(k_1-t-1)}{n-k_2-1\b k_1-t-1}\\
	&\ge\ q^{(k_1-t-1)(n-k_1)}
\end{align*}
and 
$$|\mf_2|\le{n-t\b k_2-t}+q{k_2+1\b1}.$$
Together with Lemma \ref{gaosifangsuo} and $n\ge k_1+k_2+4\ge\frac{k_1+2}{k_2-t}+k_2+2$, we get
\begin{align*}
	&\ a_1(n,k_1,k_2,t)-|\mf_1||\mf_2|\\
	>&\ |\mma_1(M,T;k_1,t)|{n-t\b k_2-t}-(|\mma_1(M,T;k_1,t)|-|\mr_2|)\left({n-t\b k_2-t}+q{k_2+1\b1}\right)\\
	=&\ |\mr_2|\left({n-t\b k_2-t}+q{k_2+1\b1}\right)-q{k_2+1\b1}|\mma_1(M,T;k_1,t)|\\
	\ge&\ q^{(k_1-t-1)(n-k_1)}{n-t\b k_2-t}-q{k_2+1\b1}{k_2-t+1\b1}{n-t-1\b k_1-t-1}\\
	\ge&\ q^{(k_1-t-1)(n-k_1)}\left(q^{(k_2-t)(n-k_2)}-q^{k_1+2k_2-2t+2}\right)\\
	\ge&\ 0.
\end{align*}
Then (3) holds.

	\medskip
	\noindent{\bf Case 2.  $t+2\le m\le k_2$.}
	\medskip

	Suppose that $T\vee G=H$ for each $G\in\mf_2\bs(\mf_2)_T$. Notice that 
	\begin{align*}
		\mf_2\subset\{F\in\mo(t,k_2): T\subset F\}\cup\{F\in\mo(k_2,H): T\not\subset F\}.
	\end{align*}
 and $\dim H=k_2+1$. Then $\mf_1$ and $\mf_2$ have been studied in Case 1.  Therefore, we suppose that there exists $F_{2,2}\in\mf_2\bs(\mf_2)_T$ such that $H\neq T\vee F_{2,2}$.  Observe that $\dim(T\vee F_{2,2})=k_2+1$ and $\dim(F\cap(T\vee F_{2,2}))\ge t+1$ for each $F\in\mf_1$.

	Write $W=(T\vee F_{2,2})\cap H$, $\dim W=w$ and
	$$\mmu=\left\{(U_1,U_2)\in\mo(t+1,T\vee F_{2,2})\times\mo(t+1,H):  T\subset U_i\not\subset W, i=1,2\right\}.$$
	 We have
	\begin{equation}\label{af-uu1u2}
		\mf_1\subset\left(\bigcup_{U\in\mo(t+1, W),\ T\subset U}(\mf_1)_U\right)\cup\left(\bigcup_{(U_1,U_2)\in\mmu}(\mf_1)_{U_1\vee U_2}\right).
	\end{equation}
	Let $(U_1,U_2)\in\mmu$. Since $U_1\subset T\vee F_{2,2}$ and $U_1\not\subset W$, we have $U_1\not\subset H$. Together with $T\subset U_1\cap U_2$, we get $\dim(U_1\vee U_2)=t+2$. Hence
	$$\left|\bigcup_{(U_1,U_2)\in\mmu}(\mf_1)_{U_1\vee U_2}\right|\le\left({k_2-t+1\b1}-{w-t\b1}\right)^2{n-t-2\b k_1-t-2}.$$
	Then by $|\{U\in\mo(t+1, W):  T\subset U\}|={w-t\b1}$ and \eqref{af-uu1u2}, we have
	\begin{align*}
		|\mf_1|&\le{w-t\b1}{n-t-1\b k_1-t-1}+\left({w-t\b1}-{k_2-t+1\b1}\right)^2{n-t-2\b k_1-t-2}\\
		&=x^2{n-t-2\b k_1-t-2}+x{n-t-1\b k_1-t-1}+{k_2-t+1\b1}{n-t-1\b k_1-t-1},
	\end{align*}
	where $x={w-t\b1}-{k_2-t+1\b1}$. By $n\ge k_1+k_2-t+2$, we have
	$${w-t\b1}-{k_2-t+1\b1}\ge-q^{k_2-t+1}\ge-q^{n-k_1-1}\ge-\dfrac{{n-t-1\b k_1-t-1}}{2{n-t-2\b k_1-t-2}}.$$
	Considering $w\le k_2$ and the property of quadratic functions, we have
	$$|\mf_1|\le{k_2-t\b1}{n-t-1\b k_1-t-1}+q^{2k_2-2t}{n-t-2\b k_1-t-2}.$$
	Together with Lemma \ref{gaosifangsuo}, we obtain
	\begin{equation}\label{af-upper-2-f1-size}
		\begin{aligned}
			|\mf_1|\le\left({k_2-t\b1}+q^{k_1+2k_2-2t-n}\right){n-t-1\b k_1-t-1}.
		\end{aligned}
	\end{equation}
	Note that
	$$\mf_2\subset\{F\in\mo(k_2,n):  T\subset F\}\cup\{F\in\mo(k_2,n):  T\not\subset F, \dim(F\cap M)=m-1\}.$$
	By Lemmas \ref{gaosifangsuo} and \ref{af-weishu}, we have
	\begin{align*}
		|\mf_2|&\le{n-t\b k_2-t}+q{m\b1}{n-m+1\b k_2-m+1}<{n-t\b k_2-t}+q^{m+1}{n-m+1\b k_2-m+1}.
	\end{align*}
	By $n\ge k_1+k_2\ge k_2+1$, we have
	$$\dfrac{q^{m+1}{n-m+1\b k_2-m+1}}{q^{m+2}{n-m\b k_2-m}}=\dfrac{(q^{n-m+1}-1)}{q(q^{k_2-m+1}-1)}\ge q^{n-k_2-1}\ge1.$$
	Then
	$$|\mf_2|<{n-t\b k_2-t}+q^{t+3}{n-t-1\b k_2-t-1}$$
	Together with \eqref{af-upper-2-f1-size}, Lemmas \ref{gaosifangsuo}, \ref{af-suanliangci-xiajie} and $n\ge k_1+k_2+t+7$, we get
	\begin{align*}
		&\ (a_1(n,k_1,k_2,t)-|\mf_1||\mf_2|){n-t-1\b k_1-t-1}^{-1}{n-t\b k_2-t}^{-1}\\
		>&\left({k_2-t+1\b1}-\dfrac{q(q^{k_1-t-1}-1)}{q^{n-t-1}-1}{k_2-t+1\b2}\right)\\
		&\ -\left({k_2-t\b1}+q^{k_1+2k_2-2t-n}\right)\left(1+\dfrac{q^{t+3}(q^{k_2-t}-1)}{q^{n-t}-1}\right)\\
		\ge&\ q^{k_2-t}-q^{k_1+2k_2-2t+1-n}-q^{k_1+2k_2-2t-n}-q^{2k_2+3-n}-q^{k_1+3k_2-t+3-2n}\\
		=&\ q^{k_2-t}(1-q^{k_1+k_2-t+1-n}-q^{k_1+k_2-t-n}-q^{k_2+t+3-n}-q^{k_1+2k_2+3-2n})\\
		\ge&\ q^{k_2-t}(1-q^{-2t-6}-q^{-2t-7}-q^{-k_1-4}-q^{-k_1-2t-11})\\
		>&\ 0.
	\end{align*}
Then (3) follows.

	\medskip
	\noindent{\bf Case 3. $m=t+1$.}
	\medskip
	
	In this case, let $S$ be the unique member of $\ms$.
	
	If $S$ is contained in each member of $\mf_1$, we have
	$$\mf_1\subset\mma_3(S;k_1),\quad\mf_2\subset\mma_4(S;k_2,t).$$
	Together with the maximality of $\mf_1$ and $\mf_2$, (2) follows.
	
	Now suppose that there exists $F_{1,1}\in\mf_1$ with $F_{1,1}\cap S=T$.  
	Let $I$ be a $(t+1)$-subflat of $H$ with $T\subset I\not\subset S$. Note that $I$ is not a $t$-cover of $\mf_2$. Then there exists $F_{2,3}\in\mf_2$ such that $F_{2,3}$ does not $t$-intersect $I$. Since $F_{2,3}$ is a $t$-cover of $\mf_1$, by Lemma \ref{af-fangsuo} and $n\ge k_1+k_2-t+1$, we have
	$$|(\mf_1)_I|\le{k_2-t+1\b1}{n-t-2\b k_1-t-2}.$$
	Note that $S\subset H$. Then by \eqref{af-f1-str}, we obtain
	\begin{align*}
		|(\mf_1)_T\bs(\mf_1)_S|&\le\sum_{I\in\mo(t+1,H),\ T\subset I\not\subset S}|(\mf_1)_I|\\
		&\le\left({k_2-t+1\b1}-1\right){k_2-t+1\b1}{n-t-2\b k_1-t-2}\\
		&=\ q{k_2-t\b1}{k_2-t+1\b1}{n-t-2\b k_1-t-2}.
	\end{align*}
	Together with Lemma \ref{gaosifangsuo}, we get
	\begin{equation}\label{af-upper-1-f1}
		\begin{aligned}
			|\mf_1|&=|(\mf_1)_S|+|(\mf_1)_T\bs(\mf_1)_S|\\
			&\le{n-t-1\b k_1-t-1}+q{k_2-t\b1}{k_2-t+1\b1}{n-t-2\b k_1-t-2}\\
			&< \ (1+q^{k_1+2k_2-2t+2-n}){n-t-1\b k_1-t-1}.
		\end{aligned}
	\end{equation}
	Let $T_1\in\mo(t,S)\bs\{T\}$. Since $T_1\not\subset F_{1,1}$ and $F_{1,1}$ is a $t$-cover of $\mf_2$, by Lemma \ref{af-fangsuo}, we have
	$$|(\mf_2)_{T_1}\bs(\mf_2)_S|\le\left({k_1-t+1\b1}-1\right){n-t-1\b k_2-t-1}=q{k_1-t\b1}{n-t-1\b k_2-t-1}.$$
	Then by Lemma \ref{gaosifangsuo}, we get
	\begin{align*}
		|\mf_2|&=|(\mf_2)_T|+\sum_{T_1\in\mo(t,S)\bs\{T\}}|(\mf_2)_{T_1}\bs(\mf_2)_S|\\
		&\le{n-t\b k_2-t}+q^2{t+1\b1}{k_1-t\b1}{n-t-1\b k_2-t-1}\\
		&<\ (1+q^{k_1+k_2+3-n}){n-t\b k_2-t}.
	\end{align*}
	Together with \eqref{af-upper-1-f1}, Lemmas \ref{gaosifangsuo},  \ref{af-suanliangci-xiajie} and $n\ge k_1+k_2+t+7$, we get
	\begin{align*}
		&\ (a_1(n,k_1,k_2,t)-|\mf_1||\mf_2|){n-t-1\b k_1-t-1}^{-1}{n-t\b k_2-t}^{-1}\\
		>&\left({k_2-t+1\b1}-\dfrac{q(q^{k_1-t-1}-1)}{q^{n-t-1}-1}{k_2-t+1\b2}\right)-(1+q^{k_1+2k_2-2t+2-n})(1+q^{k_1+k_2+3-n})\\
		=&\ q{k_2-t\b1}-\dfrac{q(q^{k_1-t-1}-1)}{q^{n-t-1}-1}{k_2-t+1\b2}-q^{2k_1+3k_2-2t+5-2n}-q^{k_1+2k_2-2t+2-n}-q^{k_1+k_2+3-n}\\
		>&\ q^{k_2-t}(1-q^{k_1+k_2-t+1-n}-q^{2k_1+2k_2-t+5-2n}-q^{k_1+k_2-t+2-n}-q^{k_1+t+3-n})\\
		\ge&\ q^{k_2-t}(1-q^{-2t-6}-q^{-3t-9}-q^{-2t-5}-q^{-k_2-4})\\
		>&\ 0.
	\end{align*}
	Then (3) follows.
\end{proof}

\begin{proof}[\bf Proof of Theorem \ref{non-trivial}]
	
		Let $n$, $k_1$, $k_2$ and $t$ be positive integers with $n\ge k_1+k_2+t+7$ and $k_1\ge k_2\ge t+1$.	Suppose that $\mf_1\subset\mm(k_1,n)$ and $\mf_2\subset\mm(k_2,n)$ are non-trivial cross $t$-intersecting families with maximum product of sizes. 
		
		By the definitions of $a_1(n,k_1,k_2,t)$ and $a_2(n,k_1,k_2,t)$, we have
		$$|\mf_1||\mf_2|\ge\max\{a_1(n,k_1,k_2,t),a_1(n,k_2,k_1,t),a_2(n,k_1,k_2,t),a_2(n,k_2,k_1,t)\}.$$
		Then by Lemma \ref{af-upper-other}, we have $(\tau_t(\mf_1),\tau_t(\mf_2))=(t,t+1)$ or $(\tau_t(\mf_1),\tau_t(\mf_2))=(t+1,t)$. Together with Lemma \ref{af-cf1-cf2}, \ref{af-cf-cf'} and \ref{af-upper-tt+1}, we finish the proof of Theorem \ref{non-trivial}.	
\end{proof}

\end{document}